 \documentclass[final]{opt2026}

\usepackage{mathtools}
\usepackage{microtype}
\usepackage{tikz}
\usepackage{booktabs}
\usetikzlibrary{arrows.meta}

\hypersetup{colorlinks=true,citecolor=blue!55!black,linkcolor=blue!55!black,urlcolor=blue!55!black}
\definecolor{huberblue}{HTML}{2F6B9A}
\definecolor{huberorange}{HTML}{D97706}
\definecolor{huberred}{HTML}{B42318}
\definecolor{hubergray}{HTML}{4B5563}

\newcommand{\G}{\mathcal G}
\newcommand{\F}{\mathcal F}
\newcommand{\A}{\mathcal A}

\title[Can Acceleration in Gradient-Norm Minimization Be Anytime?]
{Can Acceleration in Gradient-Norm Minimization Be Anytime?\\
Sharp Last-Iterate Limits in Smooth Convex Optimization}

\optauthor{%
\Name{Pierre Vernimmen} \Email{pierre.vernimmen@uclouvain.be}\\
\addr UCLouvain, ICTEAM/INMA, Louvain-la-Neuve, Belgium
\AND
\Name{Fran{\c{c}}ois Glineur} \Email{francois.glineur@uclouvain.be}\\
\addr UCLouvain, ICTEAM/INMA and CORE, Louvain-la-Neuve, Belgium}

\begin{document}
\maketitle

\begin{abstract}
In smooth convex optimization, the gradient norm is a directly observable measure of stationarity. Accelerating a first-order method that minimizes the gradient norm is known to be more delicate than accelerating the minimization of function values. Optimal accelerated methods such as OGM-G \citep{kimfessler2021} are known to exist for any prescribed finite horizon, but their coefficients depend explicitly on the length of that horizon, i.e. the number of iterations. 

We ask what kind of acceleration is feasible when the stopping horizon is unknown to the method, i.e. for horizon-independent methods. \citet{diakonikolaswang2022} conjectured that an $\Omega(N^{-1})$ lower bound on the squared gradient norm holds at every horizon $N$ for any nonadaptive, horizon-independent linear-span first-order method. We disprove this pointwise conjecture by exhibiting a method that achieves near-$N^{-2}$ last-iterate guarantees on a density-one set of horizons. We show instead that an $\Omega(N^{-1})$ lower bound must hold for infinitely many horizons. More precisely, if $\G_N(\A)$ denotes the bound on the squared gradient norm after $N$ iterations for a method $\A$, we prove that $\limsup_{N\to\infty}N\G_N(\A) \ge 1/2$ for any method $\A$. This bound is sharp: the constant $1/2$ is exactly attained by the horizon-independent gradient-descent schedule of \citet{rotaruglineurpatrinos2026}. 

In addition, we show that the above two extreme behaviors cannot be achieved by the same method: any method $\A$ with an $o(N^{-1})$ guarantee on a subsequence of iterates must satisfy $\limsup_N N\G_N(\A)=\infty$. In contrast, best-so-far output admits a uniform $O(N^{-2})$ guarantee.
\end{abstract}\vspace*{-.3cm}

\section{Introduction}

Some optimization runs in machine learning and other applications are not limited by a fixed compute or iteration budget, but instead stop when the last iterate achieves some convergence metric. The gradient norm is a natural stopping criterion: it is observable, does not require knowledge of $f_\star$, and measures first-order stationarity \citep{nesterov2012,diakonikolaswang2022,lanouyangzhang2026}. In this work, we focus on convergence bounds on the squared gradient norm after $N$ iterations, with respect to the objective accuracy $f(x_0)-f_\star$ of the initial iterate. The standard bound $O(N^{-1})$ of gradient descent (GD) can be accelerated: the method OGM-G \citep{kimfessler2021} achieves the optimal $O(N^{-2})$ order, but it is horizon-dependent: its coefficients explicitly depend on the number $N$ of iterations.
This paper asks whether acceleration is possible when $N$ is not known in advance.

\paragraph{Model and anytime requirement.}
We consider deterministic linear-span first-order methods, where each iterate is computed with a linear combination of all observed past gradients: \vspace*{-.2cm}
\begin{equation*}
    x_k=x_0-\frac1L\sum_{i=0}^{k-1}\beta_{i,k}\nabla f(x_i),
    \qquad k\ge1.
    \tag{LS}\label{eq:span}\vspace*{-.2cm}
\end{equation*}
Gradient descent with stepsizes $h_k$ corresponds to the case where $\beta_{i,k}=h_i$ for all $i < k$. 
A method is \emph{nonadaptive} if its coefficients are fixed before the run and independent of all observed oracle information; they may depend on known parameters and on $i,k$, and also on its prescribed number of iterations $N$ (i.e. different coefficients are used for each value of $N$). Adaptive methods are outside the scope of this work. A method is \emph{horizon-independent} (HI) when coefficients do not depend on $N$, i.e. the method is defined with a single list of coefficients $(\beta_{i,k})_{i<k}$. A guarantee is said to be \emph{anytime} if it holds at every iterate. 

Let $\F_{0,L}(\mathbb R^d)$ denote the class of convex differentiable functions admitting a minimizer and whose gradient is $L$-Lipschitz, and write $f_\star:=\min f$. For a method $\A$, define the convergence bound after $N$ iterations $\G_N(\A)$ as\vspace*{-.25cm}
\begin{equation*}
    \G_N(\A):=\sup_{\substack{d\ge1,\ f\in\F_{0,L}(\mathbb R^d),\ x_0\in\mathbb R^d, f(x_0)>f_\star}}
    \frac{\|\nabla f(x_N)\|^2}{L(f(x_0)-f_\star)}.
    \tag{Perf}\label{eq:metric}
\end{equation*}
We write $\G_N$ when $\A$ is clear. 
Note that the worst-case objective $f$ achieving the bound $\G_N$ may depend on $N$.

\paragraph{Question and contributions.}
Conjecture~1 of \citet{diakonikolaswang2022} predicts the existence of a constant 
$C>0$ such that every nonadaptive, horizon-independent method of the form \eqref{eq:span} satisfies $\G_N(\A)\ge C/N$ at every horizon $N$. We show that this pointwise statement is \emph{false}; more strongly, near-$N^{-2}$ last-iterate guarantees can hold on a density-one set. The sharp extrema are
\[
    \inf_\A\liminf_{N\to\infty}N\G_N(\A)=0 
       \quad\text{and}\quad
    \inf_\A\limsup_{N\to\infty}N\G_N(\A)=\frac12.
\]
Yet we show that no method can realize both: a finite $\limsup_{N\to\infty} N\G_N$ forces a positive value for  $\liminf_{N\to\infty}N\G_N$. Equivalently, existence of an $o(N^{-1})$ subsequence forces an infinite value for $\limsup_{N\to\infty}N\G_N$. Best-so-far output, by contrast, admits uniform $O(N^{-2})$. Table~\ref{tab:landscape} below collects these statements (where $\omega(N)$ denotes any nondecreasing function tending to infinity).
 
\begin{table}[ht!]
\centering
\small
\begin{tabular}{@{}llll@{}}
\toprule
Setting & Output & Guarantee / barrier & Source \\
\midrule
Known horizon & last iterate & $\G_N(\text{OGM-G})\le4/(N+1)^2$ &  \citet{kimfessler2021} \\
GD with $h_k>0$ & last iterate & no $o(N^{-1})$ anytime rate & \citet{tsai2026} \\
\addlinespace
HI linear span & last iterate & $\inf_\A\limsup_NN\G_N(\A)=1/2$ & \textbf{Thm.~\ref{thm:lower}}; \citet{rotaruglineurpatrinos2026} \\
\quad + finite $\limsup$ & last iterate & $\liminf_NN\G_N(\A)>0$ & \textbf{Cor.~\ref{cor:dichotomy}} (this work) \\
\quad + regular growth & last iterate & $\liminf_NN\G_N(\A)\ge1/2$ & \textbf{Prop.~\ref{prop:regular}} (this work)\\
\addlinespace
HI linear span & last iterate & $\G_N(\A)\le\omega(N)/N^2$ on a density-one set & \textbf{Thm.~\ref{thm:frequency}} (this work)\\
HI linear span & best-so-far & $\G_N^{\mathrm{best}}(\A)\le64/(N+1)^2$ & \textbf{Prop.~\ref{prop:best}} (this work)\\
\bottomrule
\end{tabular}
\caption{\small\textbf{Anytime gradient-norm guarantees and barriers for nonadaptive methods}. 
}
\label{tab:landscape}
\end{table}

\paragraph{Context and related work.}
Classical acceleration targets function values: Nesterov's method achieves the optimal $O(N^{-2})$ rate \citep{nesterov1983method}, with sharper worst-case constants given by OGM \citep{kimfessler2016,drori2017}. For gradient norms, horizon-dependent methods such as OGM-G achieve the accelerated $O(N^{-2})$ order \citep{kimfessler2018generalizing,kimfessler2021}, closely connected to Performance Estimation \citep{droriteboulle2014,taylor2017} and H-duality \citep{kimetal2023hduality}; their horizon dependence was emphasized by \citet{leeparkryu2021}. Constructions with longer stepsizes can also improve GD at prescribed horizons \citep{grimmer2024,altschulerparrilo2025,grimmershuwang2025}. 
\citet{kornowskishamir2024} then posed the anytime-acceleration question for GD in function value, subsequently answered positively by
\citet{zhangetal2025}. Recent work has also progressively strengthened lower bounds for stepsize-only acceleration in objective value, in both the known-horizon and anytime settings \citep{tsai2026,machen2026,yeliu2026,jungetal2026}. \citet{kornowskishamir2026} showed that, for (stochastic) GD in
nonsmooth Lipschitz optimization, horizon-independent stepsize schedules must incur a polylogarithmic loss in the last-iterate rate.

Closest to our setting, \citet{tsai2026} prove that nonadaptive GD with arbitrary positive stepsizes cannot achieve an $o(N^{-1})$ anytime convergence rate. Earlier, \citet{arjevanishamir2016} ruled out acceleration in objective value for the more restrictive class of nonadaptive first-order methods with time-invariant coefficients. In this work, we study the broader class of nonadaptive, horizon-independent linear-span methods for which \citet{diakonikolaswang2022} stated their conjecture, allowing unrestricted signed coefficients and access to the full gradient history.

\section{The Huber mechanism}
\label{sec:huber}

The proof of our main $\limsup$ result relies on well-chosen Huber-type functions. One-dimensional hard instances suffice since the definition of $\G_N$ in \eqref{eq:metric} takes the supremum over all dimensions. In those instances we set $x_0=1$ and, for any threshold $0<\delta\le1$, we consider $f=L\phi_\delta$, where
\begin{equation*}
    \phi_\delta(x)=
    \begin{cases}
        x^2/2,&x\le\delta,\\
        \delta x-\delta^2/2,&x\ge\delta,
    \end{cases}
    \qquad\qquad
    \phi_\delta'(x)=\min\{x,\delta\}.
    \tag{Huber}\label{eq:huber}
\end{equation*}
This one-sided Huber function is convex and $1$-smooth, with a constant-gradient plateau where $x\ge\delta$. As long as iterates remain on that plateau, their gradients equal the threshold  $\delta$, so that all linear-span methods \eqref{eq:span} collapse to $x_k=1-\delta s_k$, where the displacement $s_k$ equals the sum of all $k$-th row coefficients $s_k:=\sum_{i=0}^{k-1}\beta_{i,k}$. Using this observation, one can prove the following two lower bounds.

\begin{lemma}[Plateau]\label{lem:plateau}
Given a nonadaptive horizon-independent  linear-span method $\A$, define $B_k:=\max_{0\le i\le k}s_i$, where $s_0=0$ and $s_k:=\sum_{i=0}^{k-1}\beta_{i,k}$.
Then for any horizon $N\ge1$ the bound in \eqref{eq:metric} must satisfy $\G_N(\A)\ge\frac{1}{B_N+1/2}$.
\end{lemma}

\begin{lemma}[Threshold crossing]\label{lem:jump}
Given a nonadaptive horizon-independent  linear-span method $\A$, with $B_k$ as in Lemma~\ref{lem:plateau}, define the increment $\Delta_k:=B_k-B_{k-1}$ for all $k \ge 1$. 
Then for any horizon $N \ge 1$ for which $B_N$ sets a new record, i.e.\ when $\Delta_N>0$, the bound in \eqref{eq:metric} must satisfy $\G_N(\A) \ge\frac{(\Delta_N-1)^2}{B_{N-1}+1/2}$.
\end{lemma}
On a constant positive-gradient plateau, $B_N$ is the largest cumulative coefficient-row sum and hence the largest displacement toward the minimizer attained by time $N$. To prove Lemma~\ref{lem:plateau}, we choose \eqref{eq:huber} with threshold  $\delta=(B_N+1)^{-1}$, which gives
$x_k=1-\delta s_k\ge1-\delta B_N=\delta$ for every $k\le N$ (all iterates on the plateau); therefore $\G_N(\A) \ge\frac{\delta^2}{\delta-\delta^2/2}=\frac1{B_N+1/2}$. Likewise, when the new-record increment $\Delta_N$ is nonzero, choosing threshold 
$\delta=(B_{N-1}+1)^{-1}$ keeps all previous iterates on the plateau, while $x_N=1-\delta(B_{N-1}+\Delta_N)=\frac{1-\Delta_N}{B_{N-1}+1}$, which lies on the quadratic branch and yields the stated bound. These two choices of the Huber threshold  are illustrated in Figure~\ref{fig:huber}, and full details are given in Appendix~\ref{app:huberproofs}. For positive-stepsize GD, we have $B_N=\sum_{k<N}h_k$ and $\Delta_N=h_{N-1}$, recovering the Huber-type bounds of \citet{tsai2026}.

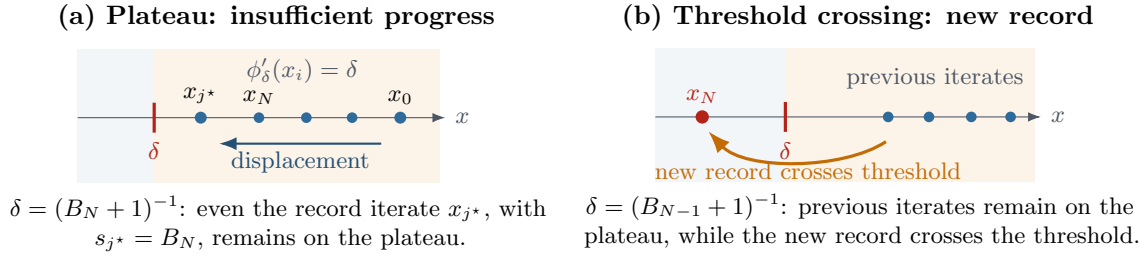
\begin{figure}[ht!]
\centering
\begin{minipage}[t]{0.495\linewidth}
\centering
{\small \textbf{(a) Plateau: insufficient progress}}\par\vspace{2mm}
\begin{tikzpicture}[x=1.10cm,y=.86cm,>=Latex,font=\footnotesize]
  \fill[huberblue!6] (0,-1.00) rectangle (.92,1.05);
  \fill[huberorange!9] (.92,-1.00) rectangle (4.35,1.05);
  \draw[->,hubergray] (0,0) -- (4.42,0) node[right] {$x$};
  \draw[huberred,very thick] (.92,-.23) -- (.92,.23);
  \node[below,huberred] at (.92,-.23) {$\delta$};
  \fill[huberblue] (3.88,0) circle (2.25pt);
  \node[above] at (3.88,.06) {$x_0$};
  \foreach \x in {3.30,2.75}{\fill[huberblue] (\x,0) circle (2.0pt);}
  \fill[huberblue] (2.18,0) circle (2.0pt);
  \node[above] at (2.18,.06) {$x_N$};
  \fill[huberblue] (1.48,0) circle (2.25pt);
  \node[above] at (1.48,.06) {$x_{j^\star}$};
  \node[text=hubergray] at (2.70,.72) {{$\phi_\delta'(x_i)=\delta$}};
  \draw[->,huberblue!75!black,thick] (3.65,-.40) -- (1.70,-.40);
  \node[text=huberblue!75!black] at (2.68,-.68) {displacement};
\end{tikzpicture}

\par\vspace{0.5mm}\footnotesize
$\delta=(B_N+1)^{-1}$: even the record iterate $x_{j^\star}$, with $s_{j^\star}=B_N$, remains on the plateau.
\end{minipage}
\hfill
\begin{minipage}[t]{0.495\linewidth}
\centering
{\small \textbf{(b) Threshold crossing: new record}}\par\vspace{2mm}
\begin{tikzpicture}[x=1.00cm,y=.86cm,>=Latex,font=\footnotesize]
  \fill[huberblue!6] (0,-1.00) rectangle (1.72,1.05);
  \fill[huberorange!9] (1.72,-1.00) rectangle (5.02,1.05);
  \draw[->,hubergray] (0,0) -- (5.10,0) node[right] {$x$};
  \draw[huberred,very thick] (1.72,-.23) -- (1.72,.23);
  \node[below,huberred] at (1.72,-.23) {$\delta$};
  \foreach \x in {3.08,3.62,4.18,4.70}{\fill[huberblue] (\x,0) circle (2.0pt);}
  \node[text=hubergray] at (3.7,.6) {previous iterates};
  \fill[huberred] (.62,0) circle (2.5pt);
  \node[above,huberred] at (.62,.06) {$x_N$};
  \draw[->,huberorange!90!black,very thick]
    (3.05,-.38) .. controls (2.48,-.82) and (1.18,-.82) .. (.70,-.22);
  \node[text=huberorange!90!black] at (2.02,-.82)
    {new record crosses threshold};
\end{tikzpicture}

\par\vspace{-0.5mm}\footnotesize
$\delta=(B_{N-1}+1)^{-1}$: previous iterates remain on the plateau, while the new record crosses the threshold.
\end{minipage}
\vspace{-.2cm}
\caption{\small \textbf{The two Huber traps.}}
\label{fig:huber}
\vspace{-.2cm}
\end{figure}

\section{Last-iterate lower bound on nonadaptive horizon-independent linear-span methods}
\label{sec:barrier}

The two Huber traps expose the obstruction to an anytime rate uniformly better than $O(N^{-1})$. Making the plateau gradient small requires fast growth of $B_N$, but fast record growth creates large increments $\Delta_N$, which will trigger threshold crossing. For intuition, consider first the \emph{regular growth} regime in which we assume that $B_N/N\to a$ and $\Delta_N\to a$ for some growth parameter $a\ge0$.

\begin{proposition}[Regular record growth]
\label{prop:regular}
If the record displacements in a nonadaptive horizon-independent  linear-span method have regular growth, then $\liminf_{N\to\infty}N\G_N\ge1/2$.
\end{proposition}
Observe that the plateau bound from Lemma~\ref{lem:plateau} gives $N\G_N \ge (\frac{B_N}{N}+\frac1{2N})^{-1}$ for all $N$. When $a=0$ this gives $N\G_N\to+\infty$. When $a>0$, the two bounds from Lemmas~\ref{lem:plateau} and~\ref{lem:jump} give $\liminf_NN\G_N\ge\max \{ {1/a,(a-1)^2/a}\}$, and this last quantity is never smaller than $\frac12$. Equality occurs when $a=2$, hence $B_N\simeq2N$ is critical. GD with convergent stepsizes and fixed-coefficient Heavy Ball ($|\beta|<1$) have regular record growth, while restart schedules need not; see Appendix~\ref{app:regular}. 

In general, without the regular growth assumption, we prove a similar bound for the $\limsup$:
\begin{theorem}[Tight anytime last-iterate lower bound]\label{thm:lower}
For any nonadaptive horizon-independent linear-span method $\A$ we have
\[    \limsup_{N\to\infty}N\G_N(\A)\ge\frac12.
\]
This lower bound already holds for one-dimensional objectives.
\end{theorem}
Suppose that there exists some $\gamma<\frac12$ such that $N \G_N\le\gamma$ holds for all sufficiently large $N$. Write $u_N:=(B_N+1/2)/N$, which is a shifted running average of the record increments. Lemma~\ref{lem:jump} shows that threshold crossings cap each increment by $1+\sqrt{\gamma u_{N-1}}$. Let $a_\gamma$ be the unique positive fixed-point solution of the equation $u=1+\sqrt{\gamma u}$. Then the average $u_N$ must be pulled down, hence\vspace*{-.1cm}
\[
    \limsup_{N\to\infty}u_N
    \le a_\gamma
    :=\frac{2+\gamma+\sqrt{\gamma^2+4\gamma}}2.\vspace*{-.1cm}
\]
Since the plateau bound also implies $u_N\ge1/\gamma$ for all sufficiently large $N$, we obtain a contradiction since $a_\gamma<1/\gamma$ for every $\gamma<1/2$. A tradeoff holds between the $\limsup$ and $\liminf$ of $N \G_N(\A)$:
\begin{corollary}[Limsup--liminf tradeoff]\label{cor:dichotomy}
For any nonadaptive horizon-independent linear-span method $\A$, if $\gamma:=\limsup_{N\to\infty}N\G_N(\A)<\infty$, then $\liminf_{N\to\infty}N\G_N\ge1/a_\gamma>0$.
\end{corollary}
Contrapositively, if convergence is faster than $N^{-1}$ on some horizons, i.e. if $\G_{N_j}=o(N_j^{-1})$ along some subsequence $N_j\to\infty$, then $\liminf_NN\G_N=0$ and therefore $\limsup_NN\G_N=+\infty$. In the other direction, the smallest possible $\gamma$ is $\frac12$ and leads to $a_{1/2}=2$, so attaining the optimal finite $\limsup$ $\frac12$ forces the full sequence to satisfy $N\G_N\to \frac12$. Appendix~\ref{app:barrierproof} proves both Theorem~\ref{thm:lower} and Corollary~\ref{cor:dichotomy} through the same scalar-growth argument. In addition, the constant in Theorem~\ref{thm:lower} is tight.

Moreover, attaining the sharp $\limsup$ constant forces the record process
itself to be asymptotically rigid: $B_N/N\to2$ and
$\frac1N\sum_{k=1}^N|\Delta_k-2|\to0$; see
Corollary~\ref{cor:rigidity}.

\paragraph{Tightness.}
The constant $\frac12$ is attained by the horizon-independent GD schedule of
\citet{rotaruglineurpatrinos2026},\vspace*{-.35cm}
\[
    h_{-1}=0,\qquad
    h_k=1+\frac{2}{1+\sqrt{9-4h_{k-1}}},
    \qquad
    x_{k+1}=x_k-\frac{h_k}{L}\nabla f(x_k).
    \tag{RGP}\label{eq:rgp}\vspace*{-.3cm}
\]
Their last-iterate bound gives $\G_N\le2-h_{N-1}$, while the recurrence implies $(2-h_{N-1})^{-1}=\frac12+\sum_{k=0}^{N-1}h_k$. For GD, $B_N=\sum_{k=0}^{N-1}h_k$, so Lemma~\ref{lem:plateau} gives the reverse inequality. Thus $\G_N=2-h_{N-1}$ exactly and $N\G_N\to1/2$, as Corollary~\ref{cor:dichotomy} requires; the calculations are in Appendix~\ref{app:sharpness}.

\paragraph{Quantifiers.}
In our setup, worst-case objectives may depend on horizon $N$, as in Conjecture~1 of \citet{diakonikolaswang2022}: indeed, while Theorem~\ref{thm:lower} forces arbitrarily large bad horizons, it does not rely on one fixed hard objective. For the method~\eqref{eq:rgp}, however, the objective $f(x)=Lx^2/2$ is worst-case at every horizon; see Appendix~\ref{app:sharpness}.

\section{What remains possible without a known horizon?}
\label{sec:possibilities}

Theorem~\ref{thm:lower} rules out uniform last-iterate acceleration, but not acceleration for a subset of selected horizons. This is indeed achievable: for example, consider concatenating OGM-G runs of successive lengths $2^j$ for $j=0,1,\ldots$, each initialized from $x_0$ (i.e. these runs are fully independent). For each horizon $N_j=2^{j+1}-1$ we have $\G_{N_j}\le\frac{4}{(2^j+1)^2}=\frac{16}{(N_j+3)^2}$, so that $N_j\G_{N_j}\to0$, disproving Conjecture~1 of \citet{diakonikolaswang2022}. Related ideas based on concatenating or composing stepsize schedules appear in \cite{zhangjiang2026,grimmercompose2025}. In the above, accelerated horizons are relatively sparse, but they can be made to have density one with a more sophisticated construction. Given a set of horizons $S\subset\mathbb N$, let $\underline d(S):=\liminf_{M\to\infty}|S\cap\{1,\ldots,M\}|/M$ denote its lower asymptotic density.

\begin{theorem}[Density-one accelerated last iterates]
\label{thm:frequency}
Let $\omega:\mathbb N\to[1,\infty)$ be any nondecreasing function tending to $\infty$. There exist a nonadaptive horizon-independent  linear-span method $\A$ and a set of horizons $S\subset\mathbb N$ with lower asymptotic density equal to one such that \vspace{-.2cm}
\[
    \G_N(\A)\le\frac{\omega(N)}{N^2},
    \qquad \text{ for all } N\in S.
\]
\end{theorem}
The construction uses dyadic epochs and works by padding iterates with redundant copies (exploiting, in some sense, a weakness in the formulation of the conjecture). Choose $k_0$ such that $\omega(2^k)\ge16$ for every $k\ge k_0$. In each epoch $k\ge k_0$, run a fresh $m_k=\left\lceil {2^{k+2}}/ \sqrt{\omega(2^k)}\right\rceil$-step OGM-G scheme from $x_0$, then reproduce its endpoint for the rest of the epoch (and omit earlier epochs from $S$). The fraction of steps during which OGM-G is applied then vanishes asymptotically, while the remaining iterates, whose positions are fixed in advance, satisfy the stated bound; see Appendix~\ref{app:blocks}.

For example, choosing $\omega(N)=1+\log N$ gives $\G_N=o(N^{-1})$ on a density-one set, so Corollary~\ref{cor:dichotomy} forces $\limsup_{N\to\infty}N\G_N=+\infty$. For the explicit block construction, this divergence is witnessed at every restart: when $N=2^k$, one has $\G_N\ge\tfrac12$, and consequently $N\G_N\ge 2^{k-1}\to\infty$; see Appendix~\ref{app:blocks}.

The record process makes the mechanism explicit. After the first endpoint in each epoch, reproducing it repeats an earlier coefficient row, so $B_N$ stalls and $\Delta_N=0$. Thus the threshold-crossing trap is silent at almost every good horizon, while the rarer OGM-G and new-record horizons carry the limsup. In contrast, positive-stepsize GD cannot stall its record, since $\Delta_N=h_{N-1}>0$ at every step, explaining why a pointwise lower bound persists in the analysis of \citet{tsai2026}.

\paragraph{Uniform best-so-far acceleration.} The distinction between current and best-so-far iterates, already emphasized by \citet{kornowskishamir2024} for function values, also matters in our squared gradient norm setup. Indeed, consider a method $\A$ that, instead of returning the last iterate after $N$ iterations, returns some previously generated iterate, namely the iterate located at position $p(N)$, where the function $p$ is known in advance. In this setup, an idea similar to the one above gives uniform acceleration. Let $\G_{p(N)}$ be the bound for that output, and let $\G_N^{\mathrm{best}}$ be the bound for the best observed gradient (obtained by replacing $\|\nabla f(x_N)\|^2$ in \eqref{eq:metric} by $\min_{0\le k\le N}\|\nabla f(x_k)\|^2$).

\begin{proposition}[Uniform best-so-far acceleration]
\label{prop:best}
There exists a nonadaptive horizon-independent linear-span method $\A$ with an output selector function 
\(p:\mathbb{N}\to\mathbb{N}_0\), with \(p(N)\le N\) for every \(N\), such that \vspace*{-.2cm}
\[
    \G_N^{\mathrm{best}}(\A)\le \G_{p(N)}(\A) \le \frac{64}{(N+1)^2}, \qquad \text{for every } N\ge1.\vspace*{-.15cm}
\]
\end{proposition}
Both Theorem~\ref{thm:frequency} and Proposition~\ref{prop:best} follow from the same finite-to-anytime transfer principle outlined in Appendix~\ref{app:blocks}. For best-so-far, epochs with geometrically increasing lengths ensure that an accelerated endpoint is available at every horizon $N$; the choice of which endpoint to return depends only on $N$, not on the observed gradients.\footnote{The constant $64$ is not optimized and can be improved using non-dyadic chained blocks.}

\paragraph{Takeaway.}
When the number of iterations is known in advance, horizon-dependent methods such as OGM-G can accelerate convergence of the final gradient norm. With an unknown horizon and mandatory last-iterate output, we show that uniform anytime acceleration is impossible: we have that $\inf_\A \limsup_{N\to\infty} N \G_N(\A)=\frac12$, attained by the schedule in \eqref{eq:rgp}. Nevertheless, horizon independence allows an $o(N^{-1})$ rate at almost every horizon, but then $N\G_N$ must become arbitrarily large along the exceptional horizons. 

The requirement to output the last iterate is partially responsible for these results. Such a requirement is natural in a system that cannot retain
or return previous iterates or may be interrupted unexpectedly. Methods
that can return a previous iterate can achieve an $O(N^{-2})$ bound.

\clearpage
\bibliography{references}

\paragraph{Diclaimer} Generative-AI tools were used to assist with mathematical ideas, literature searches, language editing and critical review of the manuscript. The authors take full responsibility for the content.

\appendix

\section{Proofs of the two Huber bounds}
\label{app:huberproofs}

\begin{proof}[Proof of Lemma~\ref{lem:plateau}]
Fix $N$ and choose $\delta=(B_N+1)^{-1}$. We verify inductively that all iterates up to $N$ remain on the plateau. Since $x_0=1\ge\delta$, the claim holds initially. Now suppose that $\phi_\delta'(x_i)=\delta$ for every $i<k$. Then  $x_k=1-\delta s_k$.

Since $s_k\le B_N$,
\[
    x_k\ge1-\delta B_N
    =1-\frac{B_N}{B_N+1}
    =\frac1{B_N+1}
    =\delta.
\]
Hence $x_k$ is on the plateau and $\phi_\delta'(x_k)=\delta$, completing the induction. In particular,
\[
    \bigl|\phi_\delta'(x_N)\bigr|^2=\delta^2.
\]
Moreover, $\phi_\delta(1)-\phi_\delta(0)=\delta-\delta^2/2$. For $f=L\phi_\delta$, the factors $L^2$ cancel in \eqref{eq:metric}, and therefore
\[
    \G_N\ge\frac{\delta^2}{\delta-\delta^2/2}
    =\frac1{B_N+1/2}.
\]
\end{proof}

\begin{proof}[Proof of Lemma~\ref{lem:jump}]
Fix $N$ with $\Delta_N>0$ and choose $\delta=(B_{N-1}+1)^{-1}$. The same induction as above, using $s_k\le B_{N-1}$ for $k<N$, shows that $x_k\ge\delta$ and hence $\phi_\delta'(x_k)=\delta$ for every $k<N$.

Since $N$ sets a new record, $s_N=B_N=B_{N-1}+\Delta_N$, and therefore
\[
    x_N=1-\delta s_N=1-\frac{B_{N-1}+\Delta_N}{B_{N-1}+1}=\frac{1-\Delta_N}{B_{N-1}+1}<\delta.
\]

Thus $x_N$ lies on the quadratic branch of $\phi_\delta$, including when $x_N<0$, and $\phi_\delta'(x_N)=x_N$. Since
\[
    |\phi_\delta'(x_N)|^2=\frac{(\Delta_N-1)^2}{(B_{N-1}+1)^2},
    \qquad
    \phi_\delta(1)-\phi_\delta(0)=\frac{B_{N-1}+1/2}{(B_{N-1}+1)^2},
\]
the normalized ratio gives
\[
    \G_N\ge\frac{(\Delta_N-1)^2}{B_{N-1}+1/2}.
\]
\end{proof}

\section{Examples of regular record growth}
\label{app:regular}

We verify the two examples quoted after Proposition~\ref{prop:regular}.

\paragraph{Gradient descent with convergent stepsizes.}
For GD with positive normalized stepsizes satisfying $h_k\to h>0$, one has $B_N=\sum_{k=0}^{N-1}h_k$ and $\Delta_N=h_{N-1}$. Hence Ces\`aro convergence gives $B_N/N\to h$, while $\Delta_N\to h$, so the record process has regular growth with $a=h$. This includes constant-step GD and the schedule \eqref{eq:rgp}; for the latter, Appendix~\ref{app:sharpness} proves $h_k\to2$, so $a=2$. Together with the exact limit $N\G_N\to1/2$ proved in Section~\ref{sec:barrier}, this shows that Proposition~\ref{prop:regular} is sharp.

\paragraph{Heavy Ball.}
Consider fixed-coefficient Heavy Ball
\[
    x_{k+1}=(1+\beta)x_k-\beta x_{k-1}
    -\frac{\alpha}{L}\nabla f(x_k),
    \qquad \alpha>0,\quad |\beta|<1.
\]
After any fixed linear-span initialization, the coefficients $s_k$ satisfy $s_{k+1}=(1+\beta)s_k-\beta s_{k-1}+\alpha$. Therefore the increments $q_k:=s_k-s_{k-1}$ satisfy $q_{k+1}=\beta q_k+\alpha$, and hence
\[
    q_k\longrightarrow\frac{\alpha}{1-\beta}>0.
\]
Thus $s_k$ is eventually increasing and $s_k/k\to\alpha/(1-\beta)$. Consequently $B_N=s_N$ eventually and
\[
    \frac{B_N}{N}\to\frac{\alpha}{1-\beta},
    \qquad
    \Delta_N\to\frac{\alpha}{1-\beta},
\]
so Heavy Ball has regular record growth.

\paragraph{A weaker sufficient condition.}
Lemma~\ref{lem:plateau} alone gives $\liminf_NN\G_N\ge1/2$ whenever $\limsup_NB_N/N\le2$, without any assumption on $\Delta_N$. Thus regular growth is needed in Proposition~\ref{prop:regular} only to control average record growth above the critical growth parameter $a=2$.

\paragraph{Regularity versus a small limsup.}
Proposition~\ref{prop:regular} and Corollary~\ref{cor:dichotomy} are incomparable; neither implies the other. For a method with the optimal limsup $\gamma=1/2$, the corollary already gives $\liminf_NN\G_N\ge1/2$ with no assumption on the record process. For a regular method with a large limsup, the proposition still gives $1/2$, whereas the corollary gives only $1/a_\gamma<1/2$. Regularity and a small limsup are two independent ways of excluding the record stalling of Section~\ref{sec:possibilities}.

\paragraph{Restarted methods.}
Regular growth need not be preserved by restarting. For instance, if a fixed-coefficient momentum method is restarted periodically, its displacement increments on a constant-gradient plateau generally repeat the same finite cycle rather than converge to a single value. Then $B_N/N$ may converge while $\Delta_N$ remains periodic, so Proposition~\ref{prop:regular} does not apply in general. Nonadaptive restart schedules are nevertheless horizon-independent and nonadaptive, and hence remain within the unrestricted class of Theorem~\ref{thm:lower}; characterizing their pointwise behavior remains open.

\section{Proof of the tight last-iterate lower bound}
\label{app:barrierproof}

\begin{lemma}[Scalar growth bound]\label{lem:scalar-growth}
Let $(B_k)$ be a nonnegative nondecreasing sequence and set $\Delta_k:=B_k-B_{k-1}$. Suppose that, for some $\gamma>0$ and all sufficiently large $k$,
\[
    \Delta_k\le1+\sqrt{\frac{\gamma(B_{k-1}+1/2)}{k}}.
\]
Then
\[
    \limsup_{k\to\infty}\frac{B_k+1/2}{k}
    \le
    a_\gamma:=\frac{2+\gamma+\sqrt{\gamma^2+4\gamma}}{2}.
\]
\end{lemma}

\begin{proof}
Define
\[
    u_k:=\frac{B_k+1/2}{k}.
\]
Since $B_k=B_{k-1}+\Delta_k$ and $B_{k-1}+1/2=(k-1)u_{k-1}$,
\begin{equation}
    u_k
    =\frac{(k-1)u_{k-1}+\Delta_k}{k}
    =u_{k-1}+\frac{\Delta_k-u_{k-1}}{k}.
    \label{eq:average-slope}
\end{equation}
Thus $u_k$ is a running average of the record increments: if the new increment $\Delta_k$ is smaller than the current average $u_{k-1}$, then $u_k$ decreases. The idea is that above the fixed point $a_\gamma$, the assumed increment bound forces exactly this situation.

Choose $K\ge2$ such that the assumed increment bound holds for every $k\ge K$. For $k\ge K$, using $B_{k-1}+1/2=(k-1)u_{k-1}$, we obtain
\[
    \Delta_k
    \le1+\sqrt{\gamma\frac{k-1}{k}u_{k-1}}
    \le1+\sqrt{\gamma u_{k-1}}.
\]
Let $a_\gamma$ denote the positive solution of $a_\gamma=1+\sqrt{\gamma a_\gamma}$, namely 
\[
    a_\gamma:=\frac{2+\gamma+\sqrt{\gamma^2+4\gamma}}2.
\]
Fix $a>a_\gamma$ and define $q(u):=u-1-\sqrt{\gamma u}$. Since
\[
    q'(u)=1-\frac{\sqrt\gamma}{2\sqrt u}>0
    \qquad (u\ge a_\gamma),
\]
the function $q$ is increasing on $[a_\gamma,\infty)$. Hence, with $\eta:=q(a)>0$, whenever $u_{k-1}\ge a$,
\[
    \Delta_k
    \le1+\sqrt{\gamma u_{k-1}}
    \le u_{k-1}-\eta.
\]
Substituting into \eqref{eq:average-slope} gives
\[
    u_k\le u_{k-1}-\frac{\eta}{k}.
\]
If $u_k\ge a$ held for every $k\ge K$, summing this over $k$ would send $u_k\to-\infty$, because the harmonic series diverges. So $u_{k_0}<a$ for some $k_0\ge K$.

It remains to show that, once below $a$, the sequence cannot cross back above it. For $k\ge K$, if $u_k\le a$, then
\[
    \Delta_{k+1}
    \le1+\sqrt{\gamma\frac{k}{k+1}u_k}
    \le1+\sqrt{\gamma a}
    <a.
\]
Equation~\eqref{eq:average-slope} expresses $u_{k+1}$ as a convex combination of $u_k\le a$ and $\Delta_{k+1}<a$, so $u_{k+1}<a$. Thus $u_k\le a$ eventually.

Since this holds for every $a>a_\gamma$, letting $a\downarrow a_\gamma$ yields
\[
    \limsup_{k\to\infty}\frac{B_k+1/2}{k}
    \le a_\gamma.
\]
\end{proof}

\begin{proof}[Proof of Theorem~\ref{thm:lower}]
Suppose the conclusion were false. Then there exists $0<\gamma<1/2$ such that $\G_N\le\frac{\gamma}{N}$ for all sufficiently large $N$. Combining this with the plateau bound of Lemma~\ref{lem:plateau} gives
\[
    \frac1{B_N+1/2}\le\G_N\le\frac{\gamma}{N}, 
    \qquad \text{hence} \qquad
    \liminf_{N\to\infty}\frac{B_N+1/2}{N}\ge\frac1\gamma.
\]
Thus keeping the gradient small would require the record displacement to grow asymptotically at least as fast as $N/\gamma$.

On the other hand, the threshold-crossing bound of Lemma~\ref{lem:jump} gives, whenever $\Delta_N>0$,
\[
    \frac{(\Delta_N-1)^2}{B_{N-1}+1/2}
    \le\G_N
    \le\frac{\gamma}{N}, 
    \qquad \text{hence} \qquad
    \Delta_N
    \le
    1+\sqrt{\frac{\gamma(B_{N-1}+1/2)}{N}}.
\]
The same inequality holds trivially when $\Delta_N=0$. Lemma~\ref{lem:scalar-growth} therefore yields
\[
    \limsup_{N\to\infty}\frac{B_N+1/2}{N}
    \le
    a_\gamma
    :=
    \frac{2+\gamma+\sqrt{\gamma^2+4\gamma}}2.
\]
For $\gamma<1/2$, one has $a_\gamma<1/\gamma$. Indeed, with $q(u):=u-1-\sqrt{\gamma u}$, the identity defining $a_\gamma$ gives $q(a_\gamma)=0$, whereas
\[
    q(1/\gamma)=\frac1\gamma-2>0.
\]
Since $q$ is increasing to the right of $a_\gamma$, it follows that $a_\gamma<1/\gamma$, contradicting the preceding lower bound. Hence
\[
    \limsup_{N\to\infty}N\G_N\ge\frac12.
\]
Since the two Huber instances used in Lemmas~\ref{lem:plateau} and~\ref{lem:jump} are one-dimensional, the lower bound already holds in dimension one.
\end{proof}

\begin{proof}[Proof of Corollary~\ref{cor:dichotomy}]
Let $\gamma:=\limsup_NN\G_N<\infty$ and fix $\gamma'>\gamma$. Then $\G_N\le\gamma'/N$ for all sufficiently large $N$. As in the preceding proof,
\[
    \Delta_N\le1+\sqrt{\frac{\gamma'(B_{N-1}+1/2)}{N}}
\]
for all sufficiently large $N$, with the inequality trivial when $\Delta_N=0$. Lemma~\ref{lem:scalar-growth} therefore gives
\[
    \limsup_{N\to\infty}\frac{B_N+1/2}{N}\le a_{\gamma'}.
\]
Using Lemma~\ref{lem:plateau},
\[
    \liminf_{N\to\infty}N\G_N
    \ge\liminf_{N\to\infty}\frac{N}{B_N+1/2}
    \ge\frac1{a_{\gamma'}}.
\]
Because $\gamma^2+4\gamma=(2+\gamma)^2-4$, the constant may also be written
\[
    a_\gamma=\frac{(2+\gamma)+\sqrt{(2+\gamma)^2-4}}2,
\]
which is continuous and increasing in $\gamma$ on $(0,\infty)$. Letting $\gamma'\downarrow\gamma$ therefore proves the claim.
\end{proof}

The two limiting cases stated after Corollary~\ref{cor:dichotomy} follow. Since $1/a_\gamma>0$ for every finite $\gamma$, a vanishing liminf forces $\limsup_NN\G_N=+\infty$. And since $a_\gamma$ increases with $\gamma$ and $a_{1/2}=2$, a limsup $\gamma<1/2$ would give $\liminf_NN\G_N\ge1/a_\gamma>1/2>\gamma$, which is impossible, so the corollary recovers Theorem~\ref{thm:lower}; at $\gamma=1/2$ it returns $1/2=\gamma$, forcing $N\G_N\to1/2$.

\begin{remark}[What is not claimed]
\label{rem:notsharp}
The bound $1/a_\gamma$ is known to be attained at $\gamma=1/2$ by \eqref{eq:rgp}. For $\gamma>1/2$, it is merely the bound delivered by the two Huber traps, and we do not know whether it is sharp: we know of no method with $\limsup_NN\G_N=\gamma$ and $\liminf_NN\G_N=1/a_\gamma$. Determining the sharp tradeoff curve would be a separate result.
\end{remark}

\section{Tightness and record-growth rigidity}
\label{app:sharpness}

\subsection{Tightness of the RGP schedule}

For the positive-stepsize GD schedule of Section~\ref{sec:barrier}, the scalar record process satisfies
\[
    B_N=\sum_{k=0}^{N-1}h_k.
\]
Set $e_k:=2-h_k$. The recurrence, starting from $h_0=3/2$, gives $1<h_k<2$ for every $k\ge0$. Since $h_{-1}=0$, we have $e_{-1}=2$. The recurrence for $h_k$ gives
\[
    e_k
    =1-\frac{2}{1+\sqrt{1+4e_{k-1}}}
    =\frac{\sqrt{1+4e_{k-1}}-1}{\sqrt{1+4e_{k-1}}+1},
\]
which is equivalent to $e_{k-1}=e_k/(1-e_k)^2$. Consequently,
\[
    \frac1{e_k}-\frac1{e_{k-1}}
    =\frac{1-(1-e_k)^2}{e_k}
    =2-e_k
    =h_k.
\]
Summing from $k=0$ to $N-1$ and using $e_{-1}=2$ yields
\begin{equation}
    \frac1{2-h_{N-1}}
    =\frac12+\sum_{k=0}^{N-1}h_k.
    \label{eq:rgp-telescope}
\end{equation}
Since $h_k>1$, the right-hand side diverges, so $2-h_{N-1}\to0$ and hence $h_k\to2$. Therefore
\[
    \frac1N\left(\frac12+\sum_{k=0}^{N-1}h_k\right)\to2,
    \qquad
    N(2-h_{N-1})\to\frac12.
\]

Finally, Lemma~\ref{lem:plateau} and $B_N=\sum_{k=0}^{N-1}h_k$ give
\[
    \G_N\ge\frac1{B_N+1/2}=2-h_{N-1}.
\]
Together with the upper bound $\G_N\le2-h_{N-1}$ from \citet[Corollary~2.19]{rotaruglineurpatrinos2026}, this proves
\[
    \G_N=2-h_{N-1}
\]
for every $N\ge1$. In particular, the schedule attains the tight asymptotic constant $1/2$.

Moreover, one fixed quadratic witnesses the exact value at every horizon. For $f(x)=Lx^2/2$,
\[
    x_{k+1}=(1-h_k)x_k=-(1-e_k)x_k.
\]
Hence
\[
    \frac{\|\nabla f(x_N)\|^2}{L(f(x_0)-f_\star)}
    =2\frac{x_N^2}{x_0^2}
    =2\prod_{k=0}^{N-1}(1-e_k)^2
    =2\prod_{k=0}^{N-1}\frac{e_k}{e_{k-1}}
    =e_{N-1}
    =2-h_{N-1}.
\]
Thus the same quadratic is worst-case for every $N$.

\subsection{Rigidity at the tight constant}

Corollary~\ref{cor:dichotomy} pins the rate of an optimal method; the record process is pinned as well.

\begin{corollary}[Record-growth rigidity]\label{cor:rigidity}
If a nonadaptive linear-span method \eqref{eq:span} generated by a single infinite coefficient table $(\beta_{i,k})_{i<k}$ satisfies
\[
    \limsup_{N\to\infty}N\G_N=\frac12,
\]
then
\[
    \frac{B_N}{N}\to2,
    \qquad
    \frac1N\sum_{k=1}^N|\Delta_k-2|\to0.
\]
\end{corollary}

\begin{proof}
Fix $\gamma>1/2$. Since $\limsup_NN\G_N=1/2$, we have $\G_N\le\gamma/N$ for all sufficiently large $N$. Lemma~\ref{lem:plateau} gives
\[
    \liminf_{N\to\infty}\frac{B_N+1/2}{N}\ge\frac1\gamma,
\]
while Lemma~\ref{lem:scalar-growth} gives
\[
    \limsup_{N\to\infty}\frac{B_N+1/2}{N}
    \le
    \frac{2+\gamma+\sqrt{\gamma^2+4\gamma}}2.
\]
Letting $\gamma\downarrow1/2$, both bounds converge to $2$, so $B_N/N\to2$.

The threshold-crossing bound also gives, for all sufficiently large $N$,
\[
    \Delta_N\le1+\sqrt{\frac{\gamma(B_{N-1}+1/2)}{N}},
\]
with the inequality trivial when $\Delta_N=0$. Since $B_N/N\to2$, letting $\gamma\downarrow1/2$ yields
\[
    \limsup_{N\to\infty}\Delta_N\le2.
\]
On the other hand,
\[
    \frac1N\sum_{k=1}^N\Delta_k
    =\frac{B_N}{N}
    \longrightarrow2.
\]
Thus the increments have an asymptotic average of $2$, while their excess above $2$ vanishes. Indeed, $(\Delta_N-2)_+\to0$, so its Ces\`aro mean also tends to zero. Using
\[
    |x-2|=2(x-2)_++2-x,
\]
we obtain
\[
    \frac1N\sum_{k=1}^N|\Delta_k-2|
    =
    \frac2N\sum_{k=1}^N(\Delta_k-2)_+
    +2-\frac{B_N}{N}
    \longrightarrow0.
\]
\end{proof}

For positive-stepsize GD, $\Delta_N=h_{N-1}$, so any schedule attaining the tight constant uses normalized steps close to $2$ on all but a vanishing fraction of iterations. For a general linear-span method, the statement concerns only the scalar record process, not the full coefficient table.

\section{A general finite-to-anytime transfer principle}
\label{app:blocks}

The constructions below use two distinct consequences of the linear-span representation \eqref{eq:span}. First, a nonadaptive finite-horizon run can be embedded at arbitrary future iterations by copying its coefficient rows and assigning zero weight to unrelated gradients. By induction, these rows exactly reproduce the original finite-horizon trajectory. Consecutive embeddings alone already yield accelerated last iterates along an infinite subsequence.

Second, once an endpoint has been generated, its coefficient row can be reused at prescribed later iterations, reproducing the same point. This second operation is used only to increase the frequency of accelerated horizons. Since all block locations and coefficient rows are fixed in advance, both constructions define a single horizon-independent nonadaptive method. For example, if $y_1=x_0-\frac aL \nabla f(x_0)$ and $y_2=x_0-\frac bL \nabla f(x_0)-\frac cL\nabla f(y_1)$, then embedding $y_1,y_2$ at horizons $p,p+1$ amounts to setting $\beta_{0,p}=a$, $\beta_{0,p+1}=b$, and $\beta_{p,p+1}=c$; any later row with $\beta_{0,q}=b$ and $\beta_{p,q}=c$ reproduces $y_2$.

\begin{theorem}[Transfer of finite-horizon endpoint rates]
\label{thm:block-transfer}
Suppose that, for some $C>0$ and $p>0$, every $m\ge1$ admits a nonadaptive $m$-step linear-span method $\A^{(m)}$ satisfying
\[
    \G_m(\A^{(m)})\le\frac{C}{(m+1)^p}.
\]
Then the following hold.

\emph{(i) Best-so-far.} There exists a horizon-independent nonadaptive method $\A$ such that
\[
    \G_N^{\mathrm{best}}(\A)\le\frac{C4^p}{(N+1)^p},
    \qquad N\ge1.
\]

\emph{(ii) Fixed density.} For every $r\in(0,1)$, there exist a horizon-independent nonadaptive method $\A$ and a set $S_r\subset\mathbb N$ with $\underline d(S_r)=r$ such that
\[
    \limsup_{\substack{N\to\infty\\N\in S_r}}N^p\G_N(\A)
    \le C2^p\left(\frac{1+r}{1-r}\right)^p.
\]

\emph{(iii) Density one.} For every nondecreasing $\omega:\mathbb N\to[1,\infty)$ with $\omega(N)\to\infty$, there exist a horizon-independent nonadaptive method $\A$ and a density-one set $S$ such that
\[
    \G_N(\A)\le\frac{\omega(N)}{N^p},
    \qquad N\in S.
\]

\end{theorem}

\begin{proof}
\emph{(i) Best-so-far.} Embed $\A^{(1)},\A^{(2)},\A^{(4)},\ldots$ on consecutive blocks of lengths $1,2,4,\ldots$. If $2^j\le N<2^{j+1}$ with $j\ge1$, then the preceding block, of length $2^{j-1}$, has already been completed. Its endpoint is among $x_0,\ldots,x_N$ and therefore
\[
    \G_N^{\mathrm{best}}(\A)\le\frac{C}{(2^{j-1}+1)^p}\le\frac{C4^p}{(N+1)^p},
\]
where the second inequality follows from $N+1\le2^{j+1}\le4(2^{j-1}+1)$. The case $N=1$ follows from the first block. The endpoint used depends only on $N$ and on the nonadaptive block structure, not on the observed gradient norms.

\medskip
\noindent\emph{(ii)--(iii) Frequent last iterates.} For the last-iterate constructions, partition the positive horizons into dyadic epochs
\[
    E_k:=\{2^k,\ldots,2^{k+1}-1\},
    \qquad k\ge0.
\]
Choose a preparation fraction $\varepsilon_k\in(0,1]$ and set
\[
    m_k:=\lceil\varepsilon_k2^k\rceil.
\]
During the first $m_k$ iterations of epoch $E_k$, embed the $m_k$-step method $\A^{(m_k)}$. Let $z_k$ be its endpoint. Once $z_k$ has been produced, reproduce it at every remaining iteration of the epoch. Thus each epoch has the form
\[
    E_k:\qquad
    \underbrace{\text{preparation}}_{m_k-1}
    \quad\big|\quad
    \underbrace{z_k,\ z_k,\ldots,z_k}_{2^k-m_k+1\text{ good horizons}}.
\]
Let $S$ denote the set of good horizons. For $N\in S\cap E_k$, one has $x_N=z_k$, and hence
\[
    \G_N(\A)\le\frac{C}{(m_k+1)^p}.
\]
Since $m_k+1\ge\varepsilon_k2^k$ and $N<2^{k+1}$, this yields
\begin{equation}
    \G_N(\A)\le\frac{C2^p}{\varepsilon_k^pN^p},
    \qquad N\in S\cap E_k.
    \label{eq:block-reuse}
\end{equation}

For \emph{(ii)}, take $\varepsilon_k\equiv\varepsilon\in(0,1)$. During the preparation part of an epoch, no new horizon belongs to $S$, so the running density decreases; after the endpoint is produced, every remaining horizon is good, so the density increases. The asymptotic minimum therefore occurs immediately before the endpoint of a new block is produced. Let $M_k:=2^k+m_k-2$ denote this last preparation horizon. The completed epochs $E_0,\ldots,E_{k-1}$ contain
\[
    \sum_{j=0}^{k-1}(2^j-m_j+1)
    =(1-\varepsilon)2^k+o(2^k)
\]
good horizons, whereas
\[
    M_k=(1+\varepsilon)2^k+o(2^k).
\]
Consequently,
\[
    \underline d(S)=\frac{1-\varepsilon}{1+\varepsilon}.
\]
Choosing $\varepsilon=(1-r)/(1+r)$ gives $\underline d(S)=r$, while \eqref{eq:block-reuse} gives
\[
    \limsup_{\substack{N\to\infty\\N\in S}}N^p\G_N(\A)
    \le C2^p\left(\frac{1+r}{1-r}\right)^p.
\]

Apart from the first endpoint of each epoch, the good horizons are precisely the reproduced ones. Their within-epoch fraction is asymptotically
\[
    1-\varepsilon=\frac{2r}{1+r}.
\]
Hence frequent accelerated horizons require frequent endpoint reproduction in this construction; the reproduced fraction becomes small only when the target lower density $r$ does.

For \emph{(iii)}, choose $k_0$ sufficiently large that
$\omega(2^k)\ge 2^pC$ for every $k\ge k_0$. Define the method
arbitrarily on the finitely many earlier epochs and omit their horizons
from $S$. For every $k\ge k_0$, set
\[
    \varepsilon_k
    :=\left(\frac{2^pC}{\omega(2^k)}\right)^{1/p}
    \in(0,1],
    \qquad
    m_k:=\lceil\varepsilon_k2^k\rceil.
\]
Since $\omega(2^k)\to\infty$, we have $\varepsilon_k\to0$ and hence
$m_k=o(2^k)$. In fact, the total number of preparation horizons up to
epoch $k$ is also negligible:
\[
    \sum_{j=k_0}^k m_j=o(2^k).
\]
Indeed, for every $\eta>0$, one has $m_j\le\eta2^j$ for all sufficiently
large $j$, and therefore
\[
    \sum_{j=k_0}^k m_j
    \le O(1)+\eta\sum_{j=k_0}^k2^j
    \le O(1)+2\eta\,2^k.
\]
Since $\eta>0$ is arbitrary, the claim follows. Thus, for any $M\in E_k$,
\[
    |\{1,\ldots,M\}\setminus S|
    \le O(1)+\sum_{j=k_0}^k m_j
    =o(2^k)
    =o(M),
\]
and therefore $S$ has density one. Finally, if $N\in S\cap E_k$, then
\eqref{eq:block-reuse} and the monotonicity of $\omega$ give
\[
    \G_N(\A)
    \le\frac{C2^p}{\varepsilon_k^pN^p}
    =\frac{\omega(2^k)}{N^p}
    \le\frac{\omega(N)}{N^p}.
\]
\end{proof}

\paragraph{Application to OGM-G.}
By \citet[Theorem~6.1]{kimfessler2021}, OGM-G satisfies
\[
    \G_m(\A^{(m)})\le\frac{4}{(m+1)^2}
\]
for every prescribed horizon $m\ge1$. Applying Theorem~\ref{thm:block-transfer} with $C=4$ and $p=2$ proves Theorem~\ref{thm:frequency} and Proposition~\ref{prop:best}, and also gives the fixed-density refinement in part~(ii).

\paragraph{Explicit spikes at restart horizons.}
Each fresh block starts from $x_0$, so its first iterate has the form $x_1=x_0-\frac{\beta}{L}\nabla f(x_0)$. Every such one-step method satisfies $\G_1\ge1/2$. Indeed, if $\beta\le0$, the quadratic $f(x)=Lx^2/2$ gives $\G_1\ge2(1-\beta)^2\ge2$. If $\beta>0$, the same quadratic and Lemma~\ref{lem:plateau} give
\[
    \G_1\ge
    \max\left\{2(1-\beta)^2,\frac1{\beta+1/2}\right\}
    \ge\frac12,
\]
with equality at $\beta=3/2$. In the construction of Theorem~\ref{thm:frequency}, a fresh block starts at every $N=2^k$ with $k \ge k_0$, so $\G_{2^k}\ge1/2$ and therefore
\[
    2^k\G_{2^k}\ge2^{k-1}\to\infty.
\]
Thus the exceptional restart horizons explicitly realize the divergence required by Corollary~\ref{cor:dichotomy}.

\paragraph{Accelerated blocks require large record displacement.}
More generally, combining Lemma~\ref{lem:plateau} with $\G_m(\A^{(m)})\le C/(m+1)^p$ gives
\[
    B_m\ge\frac{(m+1)^p}{C}-\frac12.
\]
For OGM-G, this yields $B_m\ge(m+1)^2/4-1/2$: an accelerated finite-horizon block necessarily generates quadratic record displacement. The density-one construction alternates such preparation phases with long periods in which reproducing the endpoint stalls the record.
\end{document}